\documentclass{article}
\usepackage{amsmath}
\usepackage{amssymb}
\usepackage{amsthm}
\usepackage[margin=1in, includefoot]{geometry}
\usepackage{relsize}
\usepackage[numbers]{natbib}
\usepackage[toc,page]{appendix}
\usepackage{parskip}
\usepackage{graphicx}
\usepackage{float}
\usepackage{xcolor}
\newtheorem{theorem}{Theorem}

\newtheorem{proposition}{Proposition}
\newtheorem{lemma}{Lemma}

\title{An Interior A Priori Estimate for Solutions to Monge-Amp\`ere Equations with Right-Hand Side Close to One}
\author{T. O'Neill, Dr B. Cheng}
\date{}
\begin{document}
\maketitle
\begin{abstract}
We consider Monge-Amp\'ere equations with the right hand side function close to a constant and from a function class that is larger than any H\"older class and smaller than the Dini-continuous class. We establish an upper bound for the modulus of continuity of the solution's second derivatives. This bound depends exponentially on a quantity similar to but larger than the Dini semi-norm. We establish explicit control on the shape of the sequence of shrinking sections, hence revealing the nature of such exponential dependence.
\end{abstract}
The goal of this work is to establish a priori estimates for the Monge-Amp\`ere equation with right hand side close to the constant $1$ and of regularity classes larger than any H\"older space. Our technique is inspired by Jian \& Wang in \cite{XJW2}, but we also reconcile the issue raised in Figalli et al in \cite{AFCM1} by obtaining explicit bounds to control the shape of the sequence of shrinking sections. To do this we make use of results and techniques present in the second edition of the book of Gutierrez \cite{CG1}. We use a much more general continuity condition, which encompasses H\"older continuous functions. \\ \\
For an open, bounded, convex set $\Omega\subset \mathbb{R}^n$ suppose that a strictly convex $v \in C^2(\Omega) \cap C^{0,1}(\overline{\Omega})$ satisfies the following Monge-Amp\`ere equation
\begin{align}
     \det \mathrm{D}^2v &= f \text{    in    }\Omega \label{MA1} \\
      v &=0 \text{    on    }\partial \Omega \label{MA2} \\
     0< 1-\varepsilon \leq f &\leq 1+\varepsilon \label{MA3} 
\end{align}
with $\varepsilon <1/2$ a positive constant. We require $f$ to satisfy a continuity assumption, namely:
\begin{equation}\label{condition}
  ||f||_{\mathcal{C}^{1/2}} :=  \left(\int_0^1 \frac{(\omega_f(r))^{1/2}}{r} \mathrm{d}r\right)^2 \equiv \left| \left|\frac{\omega_f(r)}{r^2}\right|\right|_{L^{1/2}([0,1])}<\infty,
\end{equation}
 where $\omega_f(r)$ the \textit{modulus of continuity} of $f$, given by
\begin{equation*}
    \omega_f(r) := \left\{ \begin{array}{ll}
      \sup_{|x-y|<r} \left\{ |f(x) - f(y)|\right\} & |r| \leq diam(\Omega) \\
      \sup_{|x-y|<diam(\Omega)} \left\{ |f(x) - f(y)|\right\} &\text{else}
      \end{array}\right.
\end{equation*} 
Apparently $\omega_f(r) \leq 2\varepsilon < 1$ for every $ r>0$. Note that $||\cdot||_{\mathcal{C}^{1/2}}$ is not a norm, merely a semi-norm, but we use the notation here for convenience. \\ \\
Also note that, although $f \in C^{0,\alpha}(\Omega)$ with $\alpha \in (0,1)$ will satisfy \eqref{condition}, with $\omega_f(r) \lesssim r^\alpha$, and hence $||f||_{\mathcal{C}^{1/2}} \lesssim (2/\alpha)^2$, this condition is more general than H\"older continuity. Suppose $f:(-1,1)\rightarrow \mathbb{R}$ is given by
\begin{equation*}
  f(x) = \left\{ \begin{array}{ll}
      0 & x \in (-1,0) \\
      \frac{1}{(-\ln(x))^{3}} &x \in[0,e^{-4}) \\
      \frac{1}{64} & x \in [e^{-4},1)
\end{array} 
\right. 
\end{equation*}
Then $f$ will have modulus of continuity of 
\begin{equation*}
  \omega_f(r) = \left\{ \begin{array}{ll}
      \frac{1}{(-\ln(r))^{3}} &r \in(0,e^{-4}) \\
       \frac{1}{64} &r \in [e^{-4},1)
\end{array} 
\right. 
\end{equation*}
Clearly $f$ cannot be H\"older continuous as 
\begin{equation*}
    \lim_{r\rightarrow 0^+}\frac{\omega_f(r)}{r^\alpha} = +\infty
\end{equation*}
for any $\alpha >0$. However, $f$ satisfies \eqref{condition}:
\begin{align*}
    \int_0^{e^{-4}}\frac{1}{(-\ln(r))^{3/2}}\frac{1}{r}\mathrm{d}r &= \int_{-\infty}^{-4}\frac{1}{(-z)^{3/2}}\mathrm{d}z <\infty.
\end{align*}
We now state our main result: \\
\begin{theorem}\label{mainresult}
Let $v \in C^2(\Omega) \cap C(\overline{\Omega})$ be a (strictly) convex solution of \eqref{MA1}, \eqref{MA2} in a convex set $\Omega$ with right hand side $f$ satisfying \eqref{MA3} and \eqref{condition} with $\varepsilon \leq \varepsilon_0$ which is a positive constant that only depends on n. Then, when restricted on $\Omega' \subset \subset \Omega$, every component of $\mathrm{D}^2v$ has modulus of continuity bounded as
\begin{equation}\label{mainestimate}
    \omega_{\mathrm{D}^2 v}(d) \leq C_0 K\left(K^{1/2} d + \int_0^{Kd} \frac{\omega_f(r)}{r}\mathrm{d}r+ K^{3/2}d \int_{Kd}^K\frac{\omega_f(r)}{r^2}\mathrm{d}r\right).
\end{equation}
with $0< d < 1$, 
\begin{equation}\label{mainconstant}
K := C_1\exp\left(C_2||f||^{1/2}_{\mathcal{C}^{1/2}}\right),
\end{equation}
and constants $C_0,C_1,C_2$ depending only on $n$, $\varepsilon_0$, $\Omega$ and dist$(\Omega', \partial \Omega)$.
\end{theorem}
This result is different to that of \cite[Theorem 1]{XJW2}, with details included in our proof to address some of the issues of the final estimate presented by Figalli, Jhaveri \& Mooney in \cite{AFCM1}. The estimate in \cite[Theorem 1]{XJW2} would provide us with a linear dependence on $||f||_{C^\alpha}$ when applied to the particular case of the right hand side being H\"older continuous, something that is shown to not be the case in \cite[Theorem 1.3]{AFCM1}. The main issue is as follows: Proofs of \cite{XJW2} require a sequence of affine (forward) transformations applied to a sequence of so-called sections (see \eqref{section}), each of which is followed by a corresponding reverse transformation. Each forward transformation reduces a given section's eccentricity (loosely speaking, the section is of convex shape and is approximated as an ellipsoid); then the eccentricity of the smaller section next in the sequence is estimated. Since the ``long" and ``short" axis of the smaller section do not necessarily align with those of the bigger section, when the reverse transformation brings the bigger section back to its original shape, the eccentricity of the smaller section may be further increased. We resolve this issue by accumulating the eccentricity of these affine transformations, and establish explicit bounds on them. \\ \\
The paper is organised as follows: We begin by stating classical results concerning Monge-Amp\`ere equations, before refining some results concerning sections of Monge-Amp\`ere solutions with right hand side close to $1$ in $L^\infty$ norm in Section \ref{sections}, with techniques inspired by Gutierrez in \cite{CG1}. We also prove a result considering the geometry of these sections, and how the closeness of the right hand side to $1$ gives us appropriate control of the eccentricity of these sections. In other words, we may use them in a similar way that Euclidean balls are used for estimates of the Poisson equation (see, for example, \cite{GTB}). We then give a proof of our main result in Section \ref{BigProof}, using techniques from \cite{XJW2} and also using results inspired by \cite{DM1} concerning shrinking sections of Monge-Amp\`ere solutions. \\ \\
\newpage
\section{Prior Results on Monge-Amp\`ere Equations}\label{Knownstuff}
We begin with some basic properties of sections and normalised convex sets, starting with John's Lemma:
\begin{lemma}[John's Lemma]
If $\Omega \subset \mathbb{R}^n$ is a bounded convex set with nonempty interior and let $E$ be the ellipsoid of minimum volume containing $\Omega$. Then
\begin{equation*}
n^{-1} E \subset \Omega \subset E,
\end{equation*}
where $n^{-1} E$ denotes the $n^{-1}$-dilation of $E$ with respect to its center. This ellipsoid is unique.
\end{lemma} 
If $E = B_1(0)$, the unit ball, we say that $\Omega$ is \textit{normalised}. For any minimum ellipsoid $E$ there exists an affine transformation $T$ such that maps $E$ to the unit ball, whence normalises $\Omega$. While many results in literature are stated under the assumption of a normalised domain,we remark, however, that the minimality of $E$ is not necessary for these arguments to work. The essential assumption is that the original domain can be affine transformed to one that contains $n^{-1}B_1(0)$ and is contained in $B_1(0)$. This has been the case since Caffarelli's strict convexity result in \cite{LC4} and interior regularity result in \cite{LC5}. \\ \\
Any two balls that sandwich the domain in this fashion will be respectively called the ``inner" and ``outer" balls from now on. In fact, thanks to the affine invariance of the Monge-Amp\`ere equation, it is the ratio of the radii of the inner and outer balls that really matters. \\ \\
Another important result that we'll need here is the Alexandrov Maximum Principle, applied in this case to smooth solutions $u$:
\begin{lemma}[Alexandrov Maximum Principle]
If $\Omega \subset \mathbb{R}^n$ is a bounded, open and convex set and $u \in C^2(\Omega)\cap C(\overline{\Omega})$ is convex with $u = 0$ on $\partial \Omega$ then
\begin{equation}\label{Alexandrov}
|u(x_0)|^n \leq C diam(\Omega)^{n-1}dist(x_0,\partial \Omega)\int_\Omega \det\mathrm{D}^2u(x)\mathrm{d}x,
\end{equation}
for all $x_0 \in \Omega$, with constant $C$ depending only on $n$.
\end{lemma}
We also have the following comparison principle:
\begin{lemma}[Comparison Principle]
If $\Omega \subset \mathbb{R}^n$ is a bounded, open and convex domain and $u,v \in C^2(\Omega)\cap C(\overline{\Omega})$ are convex with $u\geq v$ on $\partial \Omega$. If
\begin{equation*}
    \det \mathrm{D}^2u(x) \leq \det \mathrm{D}^2v(x) \qquad \text{in}\quad \Omega,
\end{equation*}
then $u(x) \geq v(x)$ in $\Omega$. 
\end{lemma}
\subsection{Monge-Amp\`ere Equations with Constant Right Hand Side}
We now discuss some known results concerning the Monge-Amp\`ere equation with $f\equiv 1$. We begin this subsection with two lemmas from \cite{XJW2}. The first compares higher derivatives of two different solutions of \eqref{MAconstant} sufficiently ``close" together, and the second provides us with an a priori $C^4$ estimate for solutions to \eqref{MAconstant}:
\begin{lemma} \label{Lowerorderestimates}
Let $w_1, w_2$ be two convex solutions of $\det \mathrm{D}^2w_l = 1$ in $\Omega$. Suppose $||w_l||_{C^4(\Omega)} \leq C_0$. Then, if $|w_1 - w_2| \leq \delta$ in $\Omega$ for some constant $\delta >0$ we have, for $i = 1,2,3$,
\begin{equation*} 
|\mathrm{D}^i (w_1 - w_2)| \leq C\delta \quad \text{   in   } \Omega',
\end{equation*}
with the constant $C>0$ depending only on $n$ and $\text{dist}(\Omega',\partial\Omega)$.
\end{lemma} 
\begin{lemma} \label{C4Bound}
Let $\Omega$ be a bounded convex domain in $\mathbb{R}^n$. Let $w$ be a convex solution of 
\begin{align}
  \det \mathrm{D}^2w &= 1 \text{   in   }\Omega \label{MAConst1}\\
  w&=0 \text{   on   }\partial \Omega. \label{MAConst2}
\end{align} 
If $B_{n^{-1}}(0) \subset \Omega \subset B_1(0)$, then for any $\Omega' \subset \subset \Omega$, there is a constant $C >0$ depending only on $n$ and $dist(\Omega',\partial\Omega)$ such that
\begin{equation*}
\label{C4est}
||w||_{C^4(\Omega')} \leq C.
\end{equation*}
\end{lemma}
For the final part of our estimates we also need the following result from \cite{CG1}:
\begin{lemma}\label{C3constant}
Let $\Omega$ be a convex domain such that
\begin{equation*}
    B_{R_1}(0) \subset \Omega \subset B_{R_2}(0)
\end{equation*}
with $n^{-1} \leq R_1 \leq R_2 \leq 1$. Suppose that $w$ is a smooth solution to \eqref{MAConst1} \& \eqref{MAConst2}.
Then, for any subdomain $\Omega' \subset \Omega$ there exists a positive constant $C^*$ depending only on dist$(\Omega', \partial \Omega)$, such that
\begin{equation*}
    |\mathrm{D}^2w - I| \leq C^*(R_2+R_1)(R_2-R_1)
\end{equation*}
and
\begin{equation*}
    ||\mathrm{D}^3w||_{L^\infty(\Omega')} \leq C^*(R_2+R_1)(R_2-R_1).
\end{equation*}
\end{lemma}
We remark that the first such inequality can be seen as a consequence of the proof of \cite[Lemma 8.2.1]{CG1}.
\section{Sections of Solutions to Monge-Amp\`ere Equations}\label{sections}
Define a \textit{section} of $v$ as follows:
\begin{equation}\label{section}
    S_{h,v}(x_0) := \{ x \in \Omega: v(x) < v(x_0) + p\cdot(x-x_0) + h\},
\end{equation}
 for $h \in \mathbb{R}$ and $p \in \partial v(x_0)$, the sub-differential of $v$. Where $x_0$ is the minimum of $v$ we simply denote $S_{h,v}(x_0) := S_{h,v}$. For $v \in C^1(\Omega)$ we have that $\partial v(x_0) = \{\nabla v(x_0) \}$ and hence $p = \nabla v(x_0)$. 
We also have the following volume estimates for the section $S_{\hat{h},v}$:
\begin{lemma}\label{volumeestimate}
Let $\Omega$ be an open set, and $v$ a convex function satisfying \eqref{MA1} - \eqref{MA3} in $\Omega$. For $S_{\hat{h},v}(x_0) \subset \subset \Omega$, $\hat{h}>0$ there exist constants $0 < C_1\leq C_2$ such that
\begin{equation}\label{volest}
C_1\hat{h}^{n/2} \leq |S_{\hat{h},v}(x_0)| \leq C_2\hat{h}^{n/2}.
\end{equation}
\end{lemma}
We now claim here that the minimal volume ellipsoids $E$ cannot be too eccentric for solutions to the Monge-Amp\`ere equation:
\begin{lemma}
Suppose that $v$ satisfies \eqref{MA1} - \eqref{MA3}. Let $S_{\hat{h},v}(x_0) \subset\subset \Omega$, for $\hat{h}>0$. Then, the minimum enclosing ellipsoid of $S_{\hat{h},v}(x_0)$ from John's Lemma with major axis length $R$ and minor axis length $r$ satisfies
\begin{equation*}
\frac{R}{r} \leq C\hat{h}^{-n/2}.
\end{equation*}
\end{lemma}
\begin{proof} Upon subtraction of a linear function from $v$, we can assume $v(x)$ attains
minimum at $x_0$. Given the assumptions on $f$, we see that $v \in C^1(\Omega)$ (see \cite{LAAF2}, for instance). In particular, by a simple convexity argument (see, for instance, \cite[Lemma 3.2.1]{CG1}), for any $x \in S_{\hat{h},v}$, we have that
\begin{equation}\label{c1alphabound}
    v(x) - v(x_0) \leq C|x-x_0|,
\end{equation}
Then, since the maximum of $v$ in $S_{\hat{h},v}(x_0)$ is attained on the boundary, we have
\begin{equation}
B_{\hat{h}/C}(x_0) \subset S_{\hat{h},v}(x_0)
\end{equation}
Since $S_{\hat{h},v}(x_0) \subset E$, this means the minor axis length $r$ of $E$ is at least $2\hat{h}/C$. We also have that $vol(n^{-1}E) \leq C_2\hat{h}^{n/2}$ due to John's Lemma and Lemma \ref{volumeestimate}. Meanwhile, due to the lower bound on the minor axis length we have that $vol(E) \geq C'Rr^{n-1} \geq C{'}{'}R\hat{h}^{n-1}$ and therefore
\begin{equation*}
CR\hat{h}^{(n-1)} \leq vol(n^{-1}E)\leq C_2\hat{h}^{n/2} \end{equation*}
From this we have demonstrated that the ratio $R/r$ is uniformly bounded. 
\end{proof}
The following results are on normalised sections of solutions to Monge-Amp\`ere equations with right-hand side sufficiently close to 1 and the techniques are inspired by Caffarelli \cite{LC5} and Gutierrez \cite{CG1}. These results show us that normalised sections can be bounded between two balls that are sufficiently close together; so our solutions within these sections are sufficiently close to a quadratic polynomial. \\ \\
We have the following lemma inspired by Gutierrez \cite{CG1}, with a more streamlined proof. It also gives us a tighter gap between the inner and outer balls. Below, the scalars $\sigma,\sigma'$ are associated with such gaps.
\begin{lemma}\label{Stepk}
Suppose that a convex domain $\Omega \subset \mathbb{R}^n$ satisfies, for $0<\sigma \leq \frac{1-n^{-1}}{1+n^{-1}}$ 
$$B_{(1-\sigma)\sqrt{2}}(0) \subset \Omega \subset B_{(1+\sigma)\sqrt{2}}(0)$$
and suppose $v$ is a strictly convex function in $\Omega$ satisfying, with constant $\delta <1$,
\begin{align}
    1-\delta \leq \det \mathrm{D}^2v &\leq 1+\delta \text{    in    }\Omega \label{boundedMA} \\
    v&=0 \text{    in    }\partial \Omega.
\end{align}
Then there exist positive constants $\hat{c},\hat{c}_1$ that only depend on dimension $n$, such that for any $\mu>0$ satisfying $3\hat{c}(\sigma\mu + \delta^{1/2}) \leq \mu^{1/2}\leq\hat{c}_1$ we have
\begin{equation}\label{ballestimate}
    B_{(1 - \sigma')\sqrt{2}}(0) \subset \mu^{-1/2}\widetilde{T}S_{\mu,v} \subset B_{(1 + \sigma')\sqrt{2}}(0),
\end{equation}
where $\sigma' = \hat{c}\mu^{-1/2}(\sigma\mu + \delta^{1/2})$ satisfying $\sigma'\leq 1/3$, $\widetilde{T}x = \widetilde{A}(x - x_0)$ with $x_0$ being the point at which $v$ attains its minimum and  the positive definite matrix $\widetilde{A}$ satisfies
\begin{equation*}
    \det \widetilde{A} = 1 \quad \text{and for any} \quad\xi \in \mathbb{R}^n, \quad (1-c_4\sigma)|\xi|^2 \leq |\widetilde{A}\xi|^2 \leq (1+c_5\sigma)|\xi|^2 \quad\text{with}\quad 1-c_4\sigma>c_3,
\end{equation*}
where positive constants $c_3,c_4,c_5$ only depend on dimension $n$.
\end{lemma}
\begin{proof}[Proof of Lemma \ref{Stepk}]
For all $0<\sigma\leq \frac{1-n^{-1}}{1+n^{-1}}$, the assumption on $\Omega$ implies $\frac{2\sqrt{2}}{1+n^{-1}}B_{n^{-1}}(0) \subset \Omega \subset \frac{2\sqrt{2}}{1+n^{-1}}B_1(0)$ and by the remarks given after John's Lemma, any existing results on normalised domains can be applied to $\Omega$ up to an $O(1)$ constant multiplier. Then, in this proof, all versions of the capital letter $C$ denote universal constants that only depend on dimension $n$, which eventually implies that the carefully chosen constants $\hat{c}, \hat{c}_1, c_3,c_4,c_5$ also have such dependence. \\ \\
Let $w_0$ be the smooth convex solution to the equation
\begin{align}
\det \mathrm{D}^2w_0 &= 1 \text{   in   }\Omega \label{W0}\\
w_0 &=0 \text{   on   }\partial\Omega. \label{W0Bdry}
\end{align}
Then, from the comparison principle we have that
\begin{equation}\label{comp1}
|v(x) - w_0(x)| \leq \overline{C}\delta 
\end{equation}
for $x \in \Omega$. Let $x_1$ be the point at which $w_0$ attains its minimum. Then
\begin{equation*}
v(x_0) - w_0(x_1) \leq v(x_1) - w_0(x_1) \leq \overline{C} \delta,
\end{equation*}
as $x_0$ is the minimum point of $v$. Moreover
\begin{equation*}
v(x_0) - w_0(x_1) \geq v(x_0) - w_0(x_0) \geq - \overline{C}\delta,
\end{equation*}
as $x_1$ is the point at which $w_0$ attains its minimum. Hence we have that
$|v(x_0) - w_0(x_1)| \leq \overline{C}\delta$ and, in view of \eqref{comp1}
\begin{equation}\label{vvsw}
|w_0(x_0) - w_0(x_1)| \leq 2\overline{C}\delta.
\end{equation}
Due to \cite[Proposition 3.2.4]{CG1}, we have that $|v(x_0)| \approx \overline{C}_1$ and $|w_0(x_1)| \approx \overline{C}_2$, and by Alexandrov's Maximum Principle, $dist(x_i,\partial\Omega)\geq C_0$. Meanwhile Pogorelov's estimates (see, for instance \cite{CG1}) imply that
\begin{equation}\label{Pog}
    C_1I \leq \mathrm{D}^2w_0(x) \leq C_2I
\end{equation}
for any $x \in \Omega_0 := \{ x \in \Omega:dist(x,\partial \Omega )\geq C_0/2\}$. \\
Next, perform the following Taylor expansion
\begin{equation}\label{Tay1}
w_0(x_0) - w_0(x_1) = \frac{1}{2}(x_0-x_1)^T\cdot\mathrm{D}^2w_0\rvert_{x=x_1 + \theta(x_0-x_1)}\cdot(x_0-x_1) \quad \text{for some} \quad \theta \in (0,1).
\end{equation}
 Moreover, as $x_0,x_1 \in \Omega_0$, using \eqref{vvsw}, \eqref{Pog} \& \eqref{Tay1} we have that 
 \begin{equation}\label{x0x1}
|x_0-x_1| \leq C_3\delta^{1/2}.
 \end{equation} Then, for $i = 1,2,\ldots ,n$ we have
\begin{equation*}
\mathrm{D}_i w_0(x_0) - \mathrm{D}_iw_0(x_1) =  \mathrm{D}\mathrm{D}_iw_0\rvert_{x_1+\theta(x_0-x_1)}\cdot(x_0 - x_1), \quad \theta \in (0,1),
\end{equation*}
and we use \eqref{Pog} and \eqref{x0x1} alongside the fact that $x_1$ is the minimum point of $w_0$ to obtain
\begin{equation}\label{gradest}
|\mathrm{D}w_0(x_0)| \leq C_4\delta^{1/2}.
\end{equation}
We now diverge from \cite{CG1} and instead prove directly that
\begin{equation}\label{boundarycontainment}
(\mu^{1/2} - \hat{c}\tau)E \subset \partial S_{\mu,v} \subset (\mu^{1/2} + \hat{c}\tau)E \subset \Omega_0
\end{equation}
for $\tau = \sigma\mu + \delta^{1/2}$ and
\begin{equation*}
 E = \{y: \frac{1}{2}\langle\mathrm{D}^2w_0(x_0)(y-x_0),(y-x_0)\rangle \leq 1\}
\end{equation*} 
The dilation of $E$ is with respect to $x_0$. By \eqref{Pog}, the fact that 
\begin{equation}\label{ellipse1}
\frac{1}{2}\langle\mathrm{D}^2w_0(x_0)(x-x_0),(x-x_0)\rangle = (\mu^{1/2} + \hat{c}\tau)^2
\end{equation}
and the assumption that $3\hat{c}\tau \leq \mu^{1/2}$ we can deduce that 
\begin{equation}\label{xx0}
|x-x_0| \leq C'\mu^{1/2} \quad \text{for} \quad x \in \partial (\mu^{1/2} + \hat{c}\tau)E.
\end{equation} 
Therefore, by the fact that $dist(x_i
, \partial \Omega) \geq C_0$ and the definition of $\Omega_0$, and recalling the assumption $\mu \leq \hat{c}_1$,the last inclusion of \eqref{boundarycontainment} is valid if $C'\hat{c}_1=\frac{C_0}{2}$. \\ \\
Moreover, the assumption that $\Omega$ is tightly constrained between $B_{(1-\sigma_0)\sqrt{2}}$ and $B_{(1+\sigma)\sqrt{2}}$ means there exists a constant $\widetilde{C}$ so that $\widetilde{C}\sigma$ is the upper bound on $|\mathrm{D}^3w_0|$ obtained from Lemma \ref{C3constant} over the domain $\Omega_0$. By Taylor expansion, we then have from \eqref{ellipse1}, \eqref{xx0} that
\begin{align*}
w_0(x) - w_0(x_0) - \mathrm{D}w_0(x_0)\cdot(x-x_0) &\geq (\mu^{1/2} + \hat{c}\tau)^2 - \widetilde{C}\sigma|x-x_0|^3 \nonumber\\
&\geq (\mu  + 2\hat{c}\mu^{1/2}\tau + \hat{c}^2\tau^2) - C{'}{'}\sigma\mu^{3/2} \nonumber \\
& >(\mu  + 2\hat{c}\mu^{1/2}\tau - \hat{c}^2\tau^2) - C{'}{'}\sigma \mu^{3/2}.
\end{align*}
Here we flipped the sign of $\hat{c}^2\tau^2$ so that the strategy used in the next step can be re-used again later. Now establish a lower bound on $v(x) - v(x_0)$ using \eqref{comp1}, \eqref{gradest} and \eqref{xx0} and the above result:
\begin{align*}
v(x)-v(x_0) &= \left[v(x) - w_0(x)\right] - \left[v(x_0) - w_0(x_0)\right] + \left[w_0(x) - w_0(x_0) - \mathrm{D}w_0(x_0)\cdot(x-x_0)\right] + \mathrm{D}w_0(x_0)\cdot(x-x_0) \\
&> -2\overline{C}\delta + \mu  + 2\hat{c}\mu^{1/2}\tau - \hat{c}^2\tau^2 - C{'}{'}\sigma \mu^{3/2} - C_4C'\delta^{1/2}\mu^{1/2} \\
&= \mu + \left(\frac{2}{3}\hat{c}\mu^{1/2}\tau - 2\overline{C}\delta\right) +\hat{c}\tau\left(\frac{1}{3}\mu^{1/2} -\hat{c}\tau\right) +\mu^{1/2}\left(\hat{c}\tau - C{'}{'}\sigma\mu-C_4C'\delta^{1/2}\right) \\
\end{align*}
Since $\tau = \sigma \mu + \delta^{1/2}$ with assumption $3\hat{c}\tau \leq \mu^{1/2}$ the second bracket must be non-negative, and we have the third bracket to be nonnegative provided that $\hat{c} \geq \max\{C{'}{'},C_4C'\}$. Further, since the first bracket is greater than $\frac{2}{3}\hat{c}\mu^{1/2}\delta^{1/2}-2\overline{C}\delta^{1/2}\tau$. in view of the assumption $3\hat{c}\tau \leq \mu^{1/2}$ again, we can have it to be positive provided that $\hat{c}^2 \geq \overline{C}$. Then $v(x) - v(x_0) > \mu$. Hence $x \notin S_{\mu,v}$. We now wish to prove the opposite inclusion. Let $x \in \partial(\mu^{1/2} - \hat{c}\tau)E$. In a similar spirit to before, \eqref{ellipse1} and \eqref{xx0} still hold but with $\mu^{1/2}+\hat{c}\tau$ being replaced by $\mu^{1/2}-\hat{c}\tau$. Therefore
\begin{equation*}
w_0(x) - w_0(x_0) - \mathrm{D}w_0(x_0)\cdot(x-x_0) \leq \mu - 2\hat{c}\mu^{1/2}\tau + \hat{c}^2\tau^2 + C{'}{'}\sigma\mu^{3/2},
\end{equation*}
Once again, we estimate $v(x) - v(x_0)$:
\begin{align}
v(x)-v(x_0) &\leq 2\overline{C}\delta + \mu - 2\hat{c}\mu^{1/2}\tau +\hat{c}^2\tau^2 + C{'}{'}\mu^{3/2} + C_4C'\delta^{1/2}\mu^{1/2} \label{B1}\\
&= \mu - \left(\frac{2}{3}\hat{c}\mu^{1/2}\tau 2\overline{C}\delta\right) -\hat{c}\tau\left(\frac{1}{3}\mu^{1/2}-\hat{c}\tau\right) - \mu^{1/2}\left(\hat{c}\tau - C{'}{'}\sigma\mu - C_4C'\delta^{1/2}\right), \label{B2}
\end{align}
and exactly the same argument shows that $v(x)-v(x_0)<\mu$, and hence $(\mu^{1/2} - \hat{c}\tau)E \subset S_{\mu,v}$, giving us \eqref{boundarycontainment}. \\
For the final part of the proof, recall the definition of set $E$ to have that with $\widetilde{A}=\sqrt{\mathrm{D}^2w_0(x_0)},$
\begin{equation*}
\partial\mu^{1/2}E = x_0 + \widetilde{A}^{-1}\left(\partial B_{\sqrt{2\mu}}(0)\right) \quad \text{and} \quad \det \widetilde{A}=1,
\end{equation*}
 where dilation of set $E$ is with respect to $x_0$. Set the affine transformation $\widetilde{T}x= \widetilde{A}(x-x_0)$ and we have that
\begin{equation*}
\widetilde{T}\left((1 \pm \hat{c}\mu^{-1/2}\tau)\mu^{1/2}E\right) = B_{(1 \pm \hat{c}\mu^{-1/2}\tau)\sqrt{2\mu}}(0),
\end{equation*}
with $\tau = \sigma\mu + \delta^{1/2}$. Then, from \eqref{boundarycontainment} we have proven \eqref{ballestimate}. The estimate $\sigma'\leq 1/3$ is simply due to the assumption $3\hat{c}(\sigma\mu +\delta^{1/2})\leq \mu^{1/2}$.
\end{proof}
The bound on $\widetilde{A}$ is due to Pogorelov's Estimate.  Now, for some constant $\mu<1$, $k=1,2,\ldots$ consider the sequence of level sets $S_{\mu^k,v}$. Clearly $0 \in S_{\mu^k,v}$ for every $k \in \mathbb{N}_0$, and the set $S_{\mu^k,v} \rightarrow \{0\}$ as $k \rightarrow \infty$. We make the following claim.
\begin{lemma}\label{bigresult1}
Suppose the convex domain $\Omega \subset \mathbb{R}^n$ containing the origin satisfies $B_{(1-\sigma_0)\sqrt{2}}(\xi) \subset \Omega \subset B_{(1+\sigma_0)\sqrt{2}}(\xi)$ for $\sigma_0 = \frac{1-n^{-1}}{1+n^{-1}}$, and some $\xi \in \mathbb{R}^n$. Let $v \in C^2(\Omega)\cap C(\overline{\Omega})$ be the strictly convex solution of
\begin{align}
\det \mathrm{D}^2v&=f \text{  in   }\Omega, \\
    v&= 0 \text{  on   }\partial \Omega,
\end{align}
and suppose that $v$ attains its minimum at the origin. Let $\hat{c},\hat{c}_1$ denote the same constants as in Lemma \ref{Stepk}. Fix a positive constant $\mu<\min\{1,\hat{c}_1^2,1/(3\hat{c})^2\}$. Finally, suppose $f(0)=1$. Then there exists a constant $\hat{c}_2$ depending only on $n$ so that, if
\begin{equation}\label{delta0}
\delta_0 :=\sup_{x,y\in\Omega}\left\{|f(x)-f(y)| \right\} <\min \left\{  1,\mu\left(\frac{1}{3\hat{c}}-\frac{1-n^{-1}}{1+n^{-1}}\mu^{1/2}\right)^2, \left(\frac{\mu^{1/2}(1-\hat{c}\mu^{1/2})\ln\mu^{-1/2}}{\hat{c}_2}\right)^2  \right\},
\end{equation}
then, for
\begin{equation}\label{deltak}
    \delta_k := \sup_{x,y, \in S_{\mu^k,v}}|f(x) - f(y)|, \qquad k=1,2,\ldots
\end{equation}
and recursively defined $\sigma_k$,
\begin{equation}\label{sigmak}
\sigma_k:=\hat{c}\mu^{-1/2}(\sigma_{k-1}\mu+\delta_{k-1}^{1/2}),\qquad k=1,2,\ldots
\end{equation}
there exist a sequence of $n$-by-$n$, positive definite matrices $\{A_k\}$ satisfying, for each $k=1,2,\ldots$
\begin{align}
\det A_k = 1, \quad \prod_{i=0}^{k-1}(1-c_4\sigma_i)|x|^2 \leq |A_kx|^2 &\leq \prod_{i=0}^{k-1}(1+c_5\sigma_i)|x|^2 \quad \text{with}\quad 1-c_4\sigma_{k-1}\geq c_3, \label{Ak1} \\
B_{(1-\sigma_k)\sqrt{2}}(0) \subset \mu^{-k/2} A_kS_{\mu^k,v} &\subset B_{(1+\sigma_k)\sqrt{2}}(0), \label{Ak2}
\end{align}
for the same constants $c_3,c_4,c_5$ (that only depend on $n$) as in Lemma \ref{Stepk}. Moreover,
\begin{equation}\label{deltaomega}
    \delta_k \leq \omega_f(K\mu^{k/2}), \quad \text{and}\quad |A_kx|^2 \leq K|x|^2.
\end{equation}
for some constant $K =K\left(n,\mu,||f||_{\mathcal{C}^{1/2}}^{1/2}\right)$ independent of $k$, given by \eqref{finalconstant}.
\end{lemma}
The plan for the proof is to scale our solution $v$ such that we may apply Lemma \ref{Stepk} and then transform back, ensuring that reverting to the initial state doesn't become uncontrollable. This will roughly equate to ensuring that we have appropriate control of the parameter $\sigma$ in Lemma \ref{Stepk}. We then perform an inductive step to show that, as these sections decrease in size, we may continue to apply Lemma \ref{Stepk} on a ``re-scaled" section and then map back and look at the size of the accumulation of these normalising transformations. 
\begin{proof}[Proof of Lemma \ref{bigresult1}] First of all, we have estimate for every $k=1,2,\ldots$,
\begin{align}
3\hat{c}(\sigma_{k-1}\mu+\delta_{k-1}^{1/2}) &\leq \mu^{1/2} \leq \hat{c}_1, \label{estimate1} \\
\sigma_k &\leq \frac{1}{3} \leq \frac{1-n^{-1}}{1+n^{-1}}, \label{estimate2}
\end{align}
due to the assumptions on $\sigma_0,\mu,\delta_0$, the fact that $\delta_0\geq \delta_k$ and a straightforward induction. \\ \\
Let $\mathcal{P}_k$ for $k \in \mathbb{N}^+$ be the statement that \eqref{Ak1} - \eqref{Ak2} holds, and prove $\mathcal{P}_k$ by induction. First, when $k=1$ by \eqref{estimate1} and the definition of $\sigma_0$ we can apply Lemma \ref{Stepk}, noting that the minimum of $v$ is attained at the origin so that $\widetilde{T}x=\widetilde{A}x$, and prove $\mathcal{P}_1$ follows by defining $A_1=\widetilde{A}_1$. \\ \\
Next, suppose that $\mathcal{P}_k$ holds for some $k \in \mathbb{N}^+$. Define
\begin{equation}\label{vk}
    \widetilde{v}_k(x) := \mu^{-k}\left(v\left(\mu^{k/2} A_k^{-1}x\right)- \mu^k\right),
\end{equation}
and
\begin{equation*}
    \widetilde{\Omega}_k = \mu^{-k/2}A_k S_{\mu^k,v},
\end{equation*}
so that $\widetilde{v}_k = 0$ on $\partial \widetilde{\Omega}_k$.Since $\det A_k=1$, we have that $\det \mathrm{D}^2 \widetilde{v}_k(x) = f(\mu^{k/2} A_k^{-1}x) = \widetilde{f}_k(x)$ in $\widetilde{\Omega}_k$. We also have, from \eqref{Ak2} of $\mathcal{P}_k$,
\begin{equation*}
    B_{(1-\sigma_k)\sqrt{2}}(0) \subset \widetilde{\Omega}_k \subset B_{(1+\sigma_k)\sqrt{2}}(0),
\end{equation*}
for $\sigma_k$ given by \eqref{sigmak} and since $f(0)=1$, for $x \in \widetilde{\Omega}_k$,
\begin{equation*}
    |\widetilde{f}_k(x) - 1| = |f(\mu^{k/2} A_k^{-1}x) - 1| \leq \sup_{x,y,\in S_{\mu^k,v}}|f(x) - f(y)| = \delta_k,
\end{equation*}
We then apply Lemma \ref{Stepk}, noting the minimum of $\widetilde{v}_k$ is attained at the origin to show that there exists $\widetilde{A}_{k+1}$ such that
\begin{equation*}
\det \widetilde{A}_{k+1} = 1\quad \text{and}\qquad (1-c_4\sigma_k)|x|^2 \leq |\widetilde{A}_{k+1}x|^2 \leq (1+c_5\sigma_k)|x|^2, \quad 1-c_4\sigma_k\geq c_3
\end{equation*}
and 
\begin{equation*}
    B_{(1-\sigma_{k+1})\sqrt{2}}(0) \subset \mu^{-1/2}\widetilde{A}_{k+1} S_{\mu,\widetilde{v}_k} \subset B_{(1+ \sigma_{k+1})\sqrt{2}}(0),
\end{equation*}
where $\sigma_{k+1} = \hat{c}\mu^{-1/2}(\sigma_k\mu + \delta_k^{1/2})$. Define $A_{k+1} := \widetilde{A}_{k+1}A_k$. Scaling back using the definition of $\widetilde{\Omega}_k$ we prove $\mathcal{P}_{k+1}$ holds. \\ \\
Next, we wish to prove \eqref{deltaomega} and for that we need to show that $S_{\mu^k,v} \subset B_{K\mu^{k/2}}$ for some constant(s) $K$ under control. First, by \eqref{Ak2} we have that $S_{\mu^k,v} \subset A_k^{-1}B_{((1+ \sigma_k)\sqrt{2})\mu^{k/2}}(0)$. Subsequently, using \eqref{Ak1} we can deduce that $S_{\mu^k, v} \subset B_{\widetilde{C}_k\mu^{k/2}}(0)$, with
\begin{equation}\label{ckexp}
\widetilde{C}_k = \frac{(1+\sigma_k)\sqrt{2}}{\sqrt{\prod_{i=0}^{k-1}(1-c_4\sigma_i)}}.
\end{equation}
We have that 
\begin{equation}\label{nonuniformdelta}
\delta_k \leq \omega_f(\widetilde{C}_k\mu^{k/2})
\end{equation}
and we now wish to show that $\widetilde{C}_k$ has a uniform upper bound to complete the proof. In other words, we need to estimate $\prod_{i=1}^{k-1}(1-c_4\sigma_i)$ from below. Note from $1-c_4\sigma_i\geq c_3$ in \eqref{Ak1}, we may bound, by Taylor expansion,
\begin{equation*}
\sum_{i=0}^{k-1}\ln(1-c_4\sigma_i) \geq -c_6'\sum_{i=0}^{k-1}\sigma_i,
\end{equation*}
for some constant $c_6'$ depending only on $n$. By taking the natural log of \eqref{ckexp} and by the simple bound \eqref{estimate2}, we have that
\begin{equation}\label{lnckbound}
\ln(\widetilde{C}_k) \leq C_0 + c_6\sum_{i=0}^{k-1}\sigma_i =:\ln(C_k), \qquad k=1,2,\ldots
\end{equation}
for some constants $C_0$, $c_6$ that depend only on $n$. We now check the summability of $\sigma_i$. By an induction argument, we may re-write $\sigma_k$ as
\begin{equation}\label{sigmakpart}
\sigma_k = (\hat{c}\mu^{1/2})^k\sigma_0 + \frac{1}{\mu}\sum_{i=0}^{k-1}(\hat{c}\mu^{1/2})^{k-i}\delta_i^{1/2}, \qquad k=1,2,\ldots
\end{equation}
Then, by the assumption on $\mu$ so that $\hat{c}\mu^{1/2}<1/3$ and hence $\{(\hat{c}\mu^{1/2})^i\}$ is summable,
\begin{equation}
   \sum_{i=0}^k\sigma_i < \frac{\sigma_0}{1-\hat{c}\mu^{1/2}} + \frac{\hat{c}}{\mu^{1/2}(1-\hat{c}\mu^{1/2})}\sum_{j=0}^{k-1}\delta_j^{1/2},\label{sumbound1}
\end{equation}
We now investigate the summability of $\delta_k^{1/2}$. We have already proven the upper bounds \eqref{nonuniformdelta} and \eqref{lnckbound}. Then, by the monotonicity of modulus of continuity and $\mu<1$,
\begin{align*}
  \sum_{j=1}^{k-1}\delta_j^{1/2} &\leq \sum_{j=1}^{k-1}\omega_f^{1/2}(C_j\mu^{j/2}) <\sum_{j=1}^{k-1}\omega_f^{1/2}(C_k\mu^{j/2}) \nonumber\\
  &\leq \frac{1}{\ln \mu^{-1/2}}\sum_{j=1}^{k-1} \int_{\mu^{j/2}}^{\mu^{(j-1)/2}}\frac{\omega_f^{1/2}(C_kr)}{r}\mathrm{d}r \nonumber\\
  &\leq \frac{1}{\ln\mu^{-1/2}}\int_0^{C_k}\frac{\omega_f^{1/2}(s)}{s}\mathrm{d}s.
\end{align*}
 Now, returning to $C_k$, we combine \eqref{lnckbound} \& \eqref{sumbound1} with the above to obtain
\begin{equation}\label{Ckbound}
\ln C_k \leq C_\mu'+C_\mu\int_0^{C_k}\frac{\omega_f^{1/2}(s)}{s}\mathrm{d}s,
\end{equation}
with constants $C_\mu = \frac{\hat{c}_2/2}{\mu^{1/2}(1-\hat{c}\mu^{1/2})\ln\mu^{-1/2}}$ and $C_\mu' = C_0 + \frac{\hat{c}_3/2}{1-\hat{c}\mu^{1/2}}$ independent of $k$. Clearly $\{C_k\}$ is an increasing sequence. We then have the following:
\begin{align*}
    \int_0^{C_k}\frac{\omega_f^{1/2}(s)}{s}\mathrm{d}s &= \int_0^1\frac{\omega_f^{1/2}(s)}{s}\mathrm{d}s + \int_1^{C_k}\frac{\omega_f^{1/2}(s)}{s}\mathrm{d}s \\
    &\leq ||f||^{1/2}_{\mathcal{C}^{1/2}} + \delta_0^{1/2}\ln C_k,
\end{align*}
where the second inequality follows from the definition of $\delta_0$ and the fact that modulus of continuity is defined as constant when the argument exceeds the diameter of the domain. Combining the above with \eqref{Ckbound} we have that
\begin{equation}
\ln C_k \leq C_\mu' +C_\mu||f||_{\mathcal{C}^{1/2}}^{1/2} + C_\mu\delta_0^{1/2}\ln C_k,\label{Ckbound2}
\end{equation}
Noting that the assumption on $\delta_0$ implies $C_\mu\delta_0^{1/2}\leq1/2$, we have that $C_k$ is uniformly bounded, and by \eqref{nonuniformdelta}, \eqref{lnckbound} the first half of \eqref{deltaomega} follows, with $K$ given by 
\begin{equation}\label{finalconstant}
K := C_1\exp\left(C_2||f||_{\mathcal{C}^{1/2}}^{1/2}\right),
\end{equation}
where
\begin{equation*}
    C_1 = \hat{c}_0\exp\left(\frac{\hat{c}_3}{(1-\hat{c}\mu^{1/2})}\right), \quad C_2 = \frac{\hat{c}_2}{\mu^{1/2}(1-\hat{c}\mu^{1/2})\ln \mu^{-1/2}}
\end{equation*}
for positive constants $\hat{c}_0, \hat{c}_2,\hat{c}_3$ that only depend on $n$. The second half of \eqref{deltaomega} is due to \eqref{Ak1}, the definition of $C_k$ and the bound \eqref{finalconstant} on $C_k$ that we have just proven.
\end{proof}
We remark on the optimality of the index $1/2$ in $||f||_{\mathcal{C}^{1/2}}$ of \eqref{finalconstant}, which is in contrast to the condition of Jian \& Wang in \cite{XJW2}, namely $||f||_{\mathcal{C}^1}<\infty$. First, in view of the first inclusion of \eqref{boundarycontainment} and hence the requirement that the quantity in \eqref{B1} be less than $\mu$, it is necessary to have $-2\hat{c}\mu^{1/2}\tau +C_4C'\delta^{1/2}\mu^{1/2}\leq 0$, namely $\tau$ should be \textit{at least} of order $\delta^{1/2}$ and thus in the conclusion of Lemma \ref{Stepk}, $\sigma'$ must be at least of order $\mu^{-1/2}\delta^{1/2}$. This leads to the $1/2$ powers of various indexed $\delta$s in the proof of Lemma \ref{bigresult1} where Lemma \ref{Stepk} was applied, in particular in the recursion \eqref{sigmak} and subsequently \eqref{sigmakpart}, \eqref{sumbound1}. Since $1>\delta_0\geq \delta_1\geq \ldots$ tending to zero and it is an upper bound of $\sum_i\sigma_i$ that we were seeking, such $1/2$ power is the highest value allowed. In view of \eqref{deltaomega}, this means we cannot raise the $1/2$ power in \eqref{Ckbound2} and hence the necessity in requiring $||f||_{\mathcal{C}^{1/2}}<\infty$.
It is worth noting that the assumption that $f$ is sufficiently close to 1 here so that we may use the results of Gutierrez \cite{CG1} to keep $v$ sufficiently close to a quadratic function. If we only have the more general condition of $\lambda \leq f \leq \Lambda$ then we may have to transform the sections using a very eccentric affine transformation, and then we need the more general bounds of Figalli \& Mooney \cite{AFCM1}. \\ \\
We finish this section by stating a lemma (inspired by Maldonado \cite{DM1}) which confirms that dilated sections are well-separated. \\
\begin{lemma}\label{Separation}
Let $\Omega \subset \mathbb{R}^n$ be fixed. Let $v$ satisfy \eqref{MA1} - \eqref{MA3} in $\Omega$. Assume that there exists $x_0 \in \Omega$ with $\nabla v(x_0) = 0$ and given $h>0$ such that $S_{h,v}\subset \subset \Omega$. Let $T:\mathbb{R}^n \rightarrow \mathbb{R}^n$ be an affine transformation normalising $S_{h,w}$. Then, for every $\lambda \in (0,1)$ we have the estimate
\begin{equation}\label{separatesets}
\text{dist}(S_{\lambda h,v}, \partial S_{h,v}) \geq C(1-\lambda)^n||T||^{-1},
\end{equation}
where $||T||$ denotes the standard matrix (operator) norm of $T$.
\end{lemma} 
\begin{proof}
Define $\Omega^* = S_{h,v}$, and choose $x \in \Omega^*$. Without loss of generality, assume $v=0$ on $\partial \Omega$. \\ \\
Suppose for the moment, that $||T||\equiv 1$. We have that $\left|\min_{\Omega^*}v\right| = h \geq a_n$, the last inequality following from the comparison principle. Let $x' \in \Omega^{**}=:S_{\lambda h,v}$, $\lambda \in (0,1)$ be such that $|x-x'| = dist(S_{\lambda h,v}, \partial S_{h,v})$. By definition of $\Omega^{**}$ Alexandrov's Maximum principle, we have that $|v(x)|^n \geq (1-\lambda)^na_n^n$, and hence obtain \eqref{separatesets} for the case $||T||=1$. \\ \\
If $||T|| \neq 1$ then we define $\hat{v} = (\det T)^{2/n}v(T^{-1}x)$ and apply the same argument to $|Tx - Tx'|$. From there we note that $|x-x'| \geq ||T||^{-1}|Tx-Tx'|$, and the result follows.
\end{proof}

\section{Proof of Theorem \ref{mainresult}}\label{BigProof}
Let $w_k$, $k=1,2,\ldots$ be solutions to the following Monge-Amp\`ere equation
\begin{align}
    \det \mathrm{D}^2w_k &= f(0) = 1 \text{    in   }S_{\mu^k,v}; \label{MAconstant} \\
    w_k &= v \text{   on   }\partial S_{\mu^k,v}. \label{MAboundary}
\end{align} Consider $x \in \Omega$. By strict convexity, there exists $\hat{h} = \hat{h} (\epsilon_0, dist(x,\partial \Omega))$
so that $S_{\hat h, v}(x_0) \subset\subset \Omega$. By Lemma 8, the affine transformation that normalises the section $S_{\hat h, v}(x_0)$ has eccentricity bound that only depends on n and $\hat h (\epsilon_0, dist(x,\partial \Omega))$. Therefore, it suffices to consider a normalised domain $\Omega$ with $v=0$ on $\partial \Omega$ and the minimum of $v$ attained at the origin. Also recall the definitions of $\delta_0$, $\delta_k$ given by \eqref{delta0}, \eqref{deltak} respectively. \\
\begin{proof}[Proof of Theorem \ref{mainresult}]
We adopt the notation of Lemma \ref{bigresult1}, and as $v$ satisfies \eqref{boundedMA} with $\delta = \varepsilon$ we have that $B_{(1-\sigma_1)\sqrt{2}}(0) \subset \mu^{-1/2}A_1S_{\mu,v}\subset B_{(1+\sigma_1)\sqrt{2}}(0)$, with $\sigma_1 :=\hat{c}\mu^{-1/2}\left(\frac{1-n^{-1}}{1+n^{-1}}\mu + \varepsilon^{1/2}\right)\leq 1/3$ and $\mu$ satisfying the required assumption of Lemma \ref{Stepk}. From this we deduce that $S_{\mu^2,v}$ and $S_{\mu,v}$ are well separated by Lemma \ref{Separation} and the fact that $||A_1||^{-1}$ is bounded. We then apply Lemmas \ref{C4Bound} \& \ref{Separation} to conclude that $||w_1||_{C^4(S_{\mu^2,v})} \leq C$.
 Similarly, we have that $||w_2||_{C^4(S_{\mu^3,v})} \leq C$. \\ \\
By considering the sub-and super-solutions $(1 \mp C\delta_1)v$ respectively, we have that
\begin{equation*}
\det \mathrm{D}^2(1-C\delta_1)v \leq \det \mathrm{D}^2w_1 \leq \det\mathrm{D}^2(1+C\delta_1)v \text{  in  } S_{\mu,v}
\end{equation*}
and $w_1 = v$ on the boundary. The comparison principle yields that
\begin{equation*}
(1+C\delta_1)v \leq w_1 \leq (1-C\delta_1)v
\end{equation*}
and hence $||w_1 - v||_{L^\infty(S_{\mu,v})} \leq C\delta_1$, where $C$ is a universal constant and depends only on $n$. We can do a similar calculation to deduce that $||w_2 - v||_{L^\infty(S_{\mu^2,v})} \leq C\delta_2$, and hence $||w_1 - w_2||_{L^\infty(S_{\mu^2,v})} \leq C\delta_1$, where we note that $\delta_2 \leq \delta_1$. \\ \\
We must now compare $w_1$ and $w_2$ on appropriate sets so that we may apply Lemma \ref{Lowerorderestimates} to obtain a bound for higher order derivatives. As $\delta_1 <\delta_0$, and $\delta_0$ satisfies \eqref{delta0}, we have that $S_{\mu^2,v} \subset S_{\mu,v}$, whence $||w_1||_{C^4(S_{ \mu^2,v})} \leq C$. We can then use Lemma 5 with eccentricity bounds of Lemma \ref{bigresult1} and Lemma \ref{Lowerorderestimates} to deduce that,
\begin{equation*}
||\mathrm{D}^j(w_1 - w_2)||_{L^\infty (S_{\mu^3,v})} \leq C\delta_1, \quad j=2,3.
\end{equation*}
Here, we have used the fact that the distance from $S_{\mu^3,v}$ to $\partial S_{\mu^2,v}$ is bounded from below due to Lemma \ref{Separation} and eccentricity bounds \eqref{Ak1}, \eqref{Ak2} for the case $k=2$. \\ \\
We will now apply the above arguments to find the difference between higher derivatives of $w_k$ and $w_{k+1}$. Define $\hat{w}_k(x) := \mu^{-(k-1)}w_k(\mu^{(k-1)/2}A_{k-1}^{-1}x)$, $\hat{w}_{k+1} :=\mu^{-(k-1)}w_{k+1}(\mu^{(k-1)/2}A_{k-1}^{-1}x)$. Then, for $\widetilde{v}_{k-1}= \mu^{-(k-1)}v(\mu^{(k-1)/2}A_{k-1}^{-1}x)$, we have that
\begin{align*}
&\left\{
\begin{array}{ll}
\det \mathrm{D}^2\hat{w}_k = 1 &\text{   in   }S_{\mu,\widetilde{v}_{k-1}} \\
\hat{w}_k = \widetilde{v}_{k-1}  &\text{   on   }\partial S_{\mu,\widetilde{v}_{k-1}} \\
\end{array}
\right. \\
&\left\{
\begin{array}{ll}
\det \mathrm{D}^2\hat{w}_{k+1} = 1 &\text{   in   }S_{\mu^2,\widetilde{v}_{k-1}} \\
\hat{w}_{k+1} = \widetilde{v}_{k-1}  &\text{   on   }\partial S_{\mu^2,\widetilde{v}_{k-1}}
\end{array}
\right.
\end{align*}
Note that, by writing $\hat{t}_1 := \mu - \hat{w}_k(x_k)$, $\hat{t}_2 = \mu^2 - \hat{w}_{k+1}(x_{k+1})$, we have that $S_{\hat{t}_1,\hat{w}_k} = S_{\mu,v}$, and $S_{\hat{t}_2,\hat{w}_{k+1}} = S_{\mu^2,v}$. By applying the same argument as above we can conclude that, for $j=2,3$,
\begin{equation*}
||\mathrm{D}^j(\hat{w}_k - \hat{w}_{k+1})||_{L^\infty (S_{\mu^{3},\widetilde{v}_{k-1}})} \leq C\delta_k,
\end{equation*}
and then scaling back and applying eccentricity bounds \eqref{Ak1}, \eqref{deltaomega} of Lemma \ref{bigresult1} give us
\begin{align}
||\mathrm{D}^2w_k - \mathrm{D}^2w_{k+1}||_{L^\infty (S_{\mu^{k+2},v})} &\leq CK\delta_k \label{Est2}\\
||\mathrm{D}^3w_k - \mathrm{D}^3w_{k+1}||_{L^\infty (S_{\mu^{k+2},v})} &\leq CK^{3/2}\mu^{-(k-1)/2}\delta_k. \label{Est3}
\end{align}
Moreover,
\begin{equation}
||\mathrm{D}^2w_1(x) - \mathrm{D}^2w_{k+1}(x)||_{L^\infty (S_{\mu^{k+2},v})} \leq CK\sum_{i=0}^k \delta_i \leq CK\int_{K\mu^{k/2}}^K \frac{\omega_f(r)}{r}\mathrm{d}r.
\end{equation}
We now have that $\{w_k\}$ is a Cauchy sequence in $C^2$, and it converges to $v$ in $C^2$ (see Appendix for a proof of this). Fix $z$ near to the origin, and we now wish to estimate
\begin{align*}
|\mathrm{D}^2v(z) - \mathrm{D}^2v(0)| &\leq I_1 + I_2 + I_3 \\
&:= |\mathrm{D}^2 w_k(0) - \mathrm{D}^2 v(0)| + |\mathrm{D}^2 v(z) - \mathrm{D}^2 w_k(z)| \\ &+ |\mathrm{D}^2 w_k(z) - \mathrm{D}^2 w_k(0)|.
\end{align*}
Let $k \geq 1$ such that $\mu^{k+4} \leq v(z) \leq \mu^{k+3}$. We begin by estimating $I_1$. From \eqref{Est2}, we have that
\begin{equation}\label{I1}
    I_1 \leq C\sum_{j=k}^\infty \delta_j \leq C_1K\int_0^{K|z|}\frac{\omega_f(r)}{r}\mathrm{d}r.
\end{equation}
Now, for $I_2$, let $w_{z,l}$, $l \in \mathbb{N}$ be the solution of
\begin{align}
   \left \{ \begin{array}{ll}
   \det\mathrm{D}^2w_{z,l} &= f(z) \quad\text{   in   }\quad S_{\mu^l,v}(z) \\
    w_{z,l} &= v \quad\text{   on   }\quad\partial S_{\mu^l,v}(z)
    \end{array}
    \right.
\end{align}
Let $l_k := \inf\{l: S_{\mu^l,v}(z) \subset S_{\mu^k,v}\}$. Clearly, $l_k \geq k$.We now wish to show that $l_k \leq k + l_0$, $l_0$ some fixed constant independent of $k$. We make the dilations $x \mapsto \mu^{-k/2}x$ and $v \mapsto \mu^{-k}v$ as before, we may assume that $v(z) \leq \mu^{3}$. From a result due to Caffarelli \cite[Corollary 2]{LC4}, there exists a constant $l_0$ such that
\begin{equation*}
    v(z) - \nabla v(z)\cdot (x-z) \leq v(x) - \mu^{l_0}
\end{equation*}
for any $x \in \partial S_{\mu,v}$. From the definition of a section, we can show that $S_{\mu^{l_0},v}(z) \subset S_{\mu,v}$. Scaling back, we obtain the required result. \\ \\
We now wish to compare $w_k$ snd $w_{z,k+l_0}$. Note that, by Lemma \ref{C4Bound} and our previous claim, we have that $||w_k||_{C^4(S_{\mu^{k+l_0},v}(z))} \leq C$ and, by using our claim and a similar argument as before, we have that $||w_{z,k+l_0} - v||_{L^\infty (S_{\mu^{k+l_0},v}(z))} \leq C\delta_k$. Hence \begin{equation} \label{compoffcentre}
    ||w_{z,k+l_0} - w_k||_{L^\infty (S_{\mu^{k+l_0},v}(z))} \leq C\delta_k.
\end{equation}
We now wish to apply Lemma \ref{Lowerorderestimates} in order to obtain higher derivative estimates in the comparison between $w_k$ and $w_{z,k+l_0}$, but we need the Monge-Amp\`ere equations to have the same right hand side. We do this by multiplying \eqref{MAconstant} by $f(z)$. In other words, we transform $w_k$ by multiplying the solution by $f(z)^{1/n}$ and then trying to obtain a comparison between the modified $w_k$ and $w_{z,k+l_0}$:
\begin{equation} \label{uzuk}
    ||w_{z,k+l_0} - f(z)^{1/n}w_k||_{L^\infty(S_{\mu^{k+l_0},v}(z))} \leq ||w_{z,k+l_0} - w_k||_{L^\infty (S_{\mu^{k+l_0},v}(z))} + \left(|f(z)^{1/n} - f(0)|\right)||w_k||_{L^\infty (S_{\mu^{k+l_0},v}(z))}
\end{equation}
We may then use the definition of $\delta_k$ to deduce that $f(z)^{1/n} \leq (f(0) + \delta_k)^{1/n}$ and by Taylor expansion, for $\xi \in (0,\delta_k)$:
\begin{align*}
f(z)^{1/n} &\leq (f(0) + \delta_k)^{1/n} \\
&= f(0)^{1/n} + \delta_k\left(\frac{1}{n}(f(0) + \xi)^{1/n}\right) \\
&\leq f(0)^{1/n} + C\delta_k.
\end{align*}
From the above and \eqref{compoffcentre} we have that
\begin{equation*}
 ||w_{z,k+l_0} - f(z)^{1/n}w_k||_{L^\infty(S_{\mu^{k+l_0},v}(z))} \leq C\delta_k.
 \end{equation*}
 We then use Lemma \ref{Lowerorderestimates} to deduce that \begin{equation}\label{I21}
     |\mathrm{D}^2w_{z,k+l_0}(z) - f(z)^{1/n}\mathrm{D}^2w_k(z)| \leq CK\delta_k.
 \end{equation}
 From this and by a similar calculation to the one that follows \eqref{uzuk}, we have that
 \begin{equation}
  |\mathrm{D}^2w_{z,k+l_0}(z) - \mathrm{D}^2w_k(z)| \leq CK\delta_k. \end{equation}
In a similar way to \eqref{I1} we obtain
\begin{equation}\label{I22}
|\mathrm{D}^2v(z) - \mathrm{D}^2w_{z,k+l_0}(z)| \leq CK\sum_{j=k+l_0}^\infty \delta_j \leq C_1K\int_{0}^{K|z|}\frac{\omega_f(r)}{r}\mathrm{d}r.
\end{equation}
Combining \eqref{I21}and \eqref{I22} we obtain an estimate for $I_2$. \\ \\
Finally, we estimate $I_3$. Let $\Delta_j = w_j - w_{j-1}$. Then, by \eqref{Est3}, we have
\begin{equation*}
|\mathrm{D}^2\Delta_j(z) - \mathrm{D}^2\Delta_j(0)| \leq CK^{3/2}\mu^{-(j-2)/2}\delta_j|z|.
\end{equation*}
Hence
\begin{align*}
I_3 &\leq |\mathrm{D}^2w_{k-1}(z) - \mathrm{D}^2w_{k-1}(0)|+ |\mathrm{D}^2\Delta_k(z) - \Delta_k(0)| \\
&\leq |\mathrm{D}^2w_1(z) - \mathrm{D}^2w_1(0)| + \sum_{j=1}^k |\mathrm{D}^2\Delta_j(z) - \mathrm{D}^2\Delta_j(0)| \\
&\leq C_1K^{3/2}|z|\left( \sigma_1 + \mu\sum_{j=1}^k \mu^{-j/2}\delta_j\right) \\
&\leq C_1K^{3/2}|z|\left( \sigma_1 + K\mu\int_{K|z|}^K \frac{\omega_f(r)}{r^2} \mathrm{d}r\right),
\end{align*}
where the penultimate inequality follows from Lemma \ref{C3constant} and the containments in Corollary \ref{Stepk}. Combining our estimates for $I_1, I_2$ and $I_3$ gives us the required result.
\end{proof}

\newpage
\appendix
\section{Appendix}
Here we prove that our constant right hand side solutions $w_k$ defined on sections of $v$ converge to $v$ as $k\rightarrow \infty$:
\begin{proposition}
Let $w_k$, $v$ be solutions of systems \eqref{MAconstant}-\eqref{MAboundary} and \eqref{MA1}-\eqref{MA3} respectively. If the assumptions of Lemma \ref{bigresult1} are satisfied, then, we have that $|\mathrm{D}^2w_k(0) - \mathrm{D}^2v(0)|\rightarrow 0$ as $k\rightarrow \infty$.
\end{proposition}
\begin{proof}
First recall that the sections $S_{\mu^k,v}\rightarrow \{0\}$ as $k\rightarrow \infty$ as $0$ is the minimum point of $v$. We consider $\hat{w}_k(x)=\mu^{-k}w_k(\mu^{-k/2}A_k^{-1}x)$ and $\hat{v}_k(x) = \mu^{-k}v(\mu^{-k/2}A_k^{-1}x)$. Then, we have that
\begin{align*}
    \det\mathrm{D}^2\hat{w}_k(x) &=1 \quad\text{in }\mu^{-k/2}A_kS_{\mu^k,v} \\
    \hat{w}_k&=\hat{v}_k \quad\text{on }\partial\mu^{-k/2}A_kS_{\mu^k,v}.
\end{align*}
From Lemma \ref{bigresult1} we have that $B_{(1-\sigma_k)\sqrt{2}}(0)\subset \mu^{-k/2}A_kS_{\mu^k,v} \subset B_{(1+\sigma_k)\sqrt{2}}(0)$, and hence by Lemma \ref{C3constant} and the fact the sum of all $\sigma_k$s is finite due to the uniform bound \eqref{finalconstant} for $C_k$ which is defined in \eqref{lnckbound} we have that, as $k \rightarrow \infty$, 
\begin{equation}\label{wkid}
|\mathrm{D}^2\hat{w}_k(0)-I|\leq C\sigma_k \rightarrow 0.
\end{equation}
Moreover, as $v\in C^2(\Omega)$ we have that
\begin{equation}
|\mathrm{D}^2\hat{v}_k(x) - \mathrm{D}^2\hat{v}_k(0)| \rightarrow 0
\end{equation}
as $k\rightarrow \infty$. Next, let $\xi \in \Omega \subset \mathbb{R}^n$ be such that $|\xi|=1$, and let $\theta_{\pm}=(1\pm\sigma_k)\sqrt{2}$. We then have that
\begin{equation}
(1-\sigma_k)^2 \leq \frac{1}{2}\xi^T\mathrm{D}^2\hat{v}_k(\theta_\pm\xi)\xi \leq (1+\sigma_k)^2
\end{equation}
from a Taylor expansion and \eqref{Ak2}. Noting that this holds for arbitrary unit vector $\xi$, we must have that $|\mathrm{D}^2\hat{v}_k -I|\rightarrow 0$ as $k\rightarrow \infty$. Combining this with \eqref{wkid} gives us that
\begin{equation}
|\mathrm{D}^2\hat{w}_k - \mathrm{D}^2\hat{v}_k| \rightarrow 0 \quad \text{as}\quad k\rightarrow \infty.
\end{equation}
The result then follows by transforming back under $A_k$ and noting that it is bounded from Lemma \ref{bigresult1}.
\end{proof}

\end{document}